\documentclass[12pt,a4paper]{amsart}

\title[An Inductive Approach]%
{An Inductive Approach to Coxeter Arrangements\\ and Solomon's Descent Algebra}

\author[J.M. Douglass]{J. Matthew Douglass} 
\address{Department of Mathematics,
University of North Texas, 
Denton TX, USA 76203}
\email{douglass@unt.edu} 

\author[G. Pfeiffer]{G\"otz Pfeiffer}
\address{School of Mathematics, Statistics and Applied Mathematics,
National University of Ireland, Galway, University Road, Galway, Ireland}
\email{goetz.pfeiffer@nuigalway.ie}

\author[G. R\"ohrle]{Gerhard R\"ohrle}
\address
{Fakult\"at f\"ur Mathematik,
Ruhr-Universit\"at Bochum,
D-44780 Bochum, Germany}
\email{gerhard.roehrle@rub.de}


\usepackage[a4paper,margin=1in,includeheadfoot]{geometry}
\usepackage{parskip}
\usepackage[euler-digits]{eulervm}
\usepackage{amssymb}

\usepackage[colorlinks=true,
linkcolor=green,
filecolor=brown,
citecolor=green]{hyperref}

\usepackage{tikz}

\newcommand{\Size}[1]{\left| #1 \right|}
\newcommand{\Span}[1]{\left< #1 \right>}
\newcommand{\Floor}[1]{\lfloor #1 \rfloor}

\newcommand{\R}{\mathbb{R}}

\newcommand{\C}{\mathbb{C}}
\let\AA\relax
\newcommand{\AA}{\mathcal{A}}
\newcommand{\CC}{\mathcal{C}}

\newcommand{\LL}{\mathcal{L}}
\newcommand{\RR}{\mathcal{R}}
\renewcommand{\emptyset}{\varnothing}

\newcommand{\kW}{\C W}

\newcommand{\Fix}{\operatorname{Fix}}

\newcommand{\sh}{\operatorname{sh}}
\newcommand{\codim}{\operatorname{codim}}

\newcommand{\Ind}{\operatorname{Ind}}

\newcommand{\Res}{\operatorname{Res}}
\newcommand{\Av}{\operatorname{Av}}

\numberwithin{equation}{section}

\newtheorem{Theorem}[equation]{Theorem}
\newtheorem{Proposition}[equation]{Proposition}
\newtheorem{Lemma}[equation]{Lemma}
\newtheorem{Corollary}[equation]{Corollary}
\newtheorem{Conjecture}{Conjecture}

\theoremstyle{definition}
\newtheorem{Remark}[equation]{Remark}

\renewcommand{\labelenumi}{(\roman{enumi})}

\begin{document}

\begin{abstract}
  In our recent paper~\cite{DouglassPfeiffer2011}, we claimed that
  both the group algebra of a finite Coxeter group $W$ as well as the
  Orlik-Solomon algebra of~$W$ can be decomposed into a sum of
  induced one-dimensional representations of centralizers, one for
  each conjugacy class of elements of~$W$, and gave a uniform proof of
  this claim for symmetric groups.  In this note we outline an
  inductive approach to our conjecture.  As an application of this
  method, we prove the inductive version of the conjecture for finite
  Coxeter groups of rank up to~$2$.
\end{abstract}

\maketitle

\section{Introduction}
\label{sec:intro}

Let $W$ be a finite Coxeter group, generated by a set $S$ of simple
reflections.  If $\Size{S} = r$, then $W$ acts as a reflection group on
Euclidean $r$-space $V$.  The reflection arrangement of $W$ is the
hyperplane arrangement consisting of the reflecting hyperplanes in $V$ of
all the reflections in $W$.  The Orlik-Solomon algebra $A(W)$ of $W$ is the
cohomology ring of the complement of the complexified reflection
arrangement.  It follows from a result of Brieskorn~\cite{Brieskorn1973}
that the algebra $A(W)$ is a $W$-module of dimension $\Size{W}$. For some
history of the computation of $A(W)$ as a $W$-module, see the introduction
of our recent paper~\cite{DouglassPfeiffer2011}.

In~\cite{DouglassPfeiffer2011}, we claimed that both the group algebra $\kW$
of $W$ (affording the regular character $\rho_W$) as well as the
Orlik-Solomon algebra $A(W)$ (affording the Orlik-Solomon character
$\omega_W$) can be decomposed into a sum of induced one-dimensional
representations of centralizers, one for each conjugacy class of elements of
$W$, in the following interlaced way.
\begin{Conjecture}\label{conj:a}
  Let $\RR$ be a set of representatives of the conjugacy classes of
  $W$.  Then,
  for each $w \in \RR$, there are linear characters
  $\widetilde{\varphi}_w$ and $\widetilde{\psi}_w$ of $C_W(w)$ such
  that
  \begin{align*}
    \rho_W &= \sum_{w \in \RR} \Ind_{C_W(w)}^W \widetilde{\varphi}_{w}, &
    \omega_W &= \sum_{w \in \RR} \Ind_{C_W(w)}^W \widetilde{\psi}_{w}
  \end{align*}
  are sums of induced linear characters.  Moreover, for each $w \in
  \RR$, the characters $\widetilde{\varphi}_w$ and
  $\widetilde{\psi}_w$ can be chosen so that
  \begin{align*}
    \widetilde{\psi}_w = \widetilde{\varphi}_w \epsilon \alpha_w,
  \end{align*}
  where $\epsilon$ is the sign character of $W$, and $\alpha_w$ is the
  determinant on the $1$-eigenspace of~$w$.
\end{Conjecture}

When $W$ is a symmetric group, the formula for $\rho_W$ has been
proved independently by Bergeron, Bergeron, and
Garsia~\cite{BBGarsia1990}, Hanlon~\cite{Hanlon1990}, and
Schocker~\cite{Schocker2001}. The formula for $\omega_W$ follows from
work of Lehrer and Solomon~\cite{LehrerSolomon1986}, who also checked
the identity for $\omega_W$ in the case of a dihedral
group~$W$. Conjecture~2.1 in~\cite{DouglassPfeiffer2011} is a graded
refinement of Conjecture~\ref{conj:a} and the main result
in~\cite{DouglassPfeiffer2011} is a uniform proof of this refined
conjecture for symmetric groups.

The details of the proof of Conjecture 2.1 in~\cite{DouglassPfeiffer2011}
for symmetric groups rely on properties of these groups not shared by
other finite Coxeter groups. However, the underlying strategy of the proof
using induced characters both generalizes and admits a ``relative'' version,
for pairs $(W,W_L)$, where $W_L$ is a parabolic subgroup of $W$. In
Section~\ref{sec:claim}, we formalize this notion in Conjecture
~\ref{conj:c}, show how it leads to a proof of Conjecture ~\ref{conj:a}, and
describe a two-step procedure that can be used to prove this relative
conjecture. Prior to that, in Sections~\ref{sec:descent} and~\ref{sec:os} we
review some notation and basic facts about the descent algebra $\Sigma(W)$
and the Orlik-Solomon algebra $A(W)$. In the final section we apply the
methods from Section~\ref{sec:claim} and prove Conjecture ~\ref{conj:c} for
all pairs $(W, W_L)$ where $W$ is arbitrary and $W_L$ has rank at most
$2$. As a consequence, we deduce that Conjecture ~\ref{conj:a} holds for
Coxeter groups of rank $2$ or less.

\section{Minimal Length Transversals of Parabolic Subgroups}
\label{sec:descent}

The descent algebra of a finite Coxeter group $W$ encodes many aspects of the
combinatorics of the minimal length coset representatives of the standard
parabolic subgroups of~$W$.  In this section, we provide notation and
summarize useful properties of these distinguished coset representatives
following Pfeiffer~\cite{Pfeiffer2009}.

For $J
\subseteq S$, let
\begin{align*}
  X_J = \{w \in W : \ell(sw) > \ell(w) \text{ for all } s \in J\}.
\end{align*}
Then $X_J$ is a right transversal of the parabolic subgroup $W_J =
\Span{J}$ of $W$, consisting of the unique elements of minimal length
in their cosets.  If we set
\begin{align*}
  x_J = \sum_{x \in X_J} x^{-1} \in \kW,
\end{align*}
then, by Solomon's Theorem~\cite{Solomon1976}, the subspace
\begin{align*}
  \Sigma(W) = \Span{x_J : J \subseteq S}_{\C}
\end{align*}
is a $2^r$-dimensional subalgebra of the group algebra $\kW$, called the
descent algebra of~$W$.

For $J \subseteq S$, denote
\begin{align*}
  X_J^{\sharp} = \{x \in X_J : J^x \subseteq S\}.
\end{align*}
The action of $W$ on itself by conjugation partitions the power set of $S$
into equivalence classes of $W$-conjugate subsets.  We call the class
\begin{align*}
  [J] = \{J^x : x \in X_J^{\sharp}\}
\end{align*}
of a subset $J \subseteq S$ the \emph{shape} of $J$, and  denote by
\begin{align*}
  \Lambda = \{[J] : J \subseteq S\}
\end{align*}
the  set  of shapes  of~$W$.   The  shapes  parametrize the  conjugacy
classes  of  parabolic subgroups  of  $W$,  since  two subsets  $J,  K
\subseteq S$ are conjugate if  and only if the corresponding parabolic
subgroups $W_J$ and $W_K$ are conjugate.  We say that a parabolic subgroup
of $W$ has shape $[J]$ if it is conjugate to $W_J$ in~$W$.

Furthermore, for $J \subseteq S$, we define
\begin{align*}
  N_J = \{x \in X_J : J^x = J\}.
\end{align*}
Then $N_J$ is a subgroup of $W$ and by results of
Howlett~\cite{Howlett1980}, the normalizer of $W_J$ in $W$ is a semi-direct
product $N_W(W_J) = W_J \rtimes N_J$.

An element $w\in W$ is called \emph{cuspidal} in case $w$ has no fixed
points in the reflection representation of $W$.  For $J \subseteq S$,
an element $w \in W_J$ is cuspidal in the parabolic subgroup $W_J$ if
$w$ has no fixed points in the orthogonal complement of $\Fix(W_J)$ in
$V$.  If $w$ is a cuspidal element in $W_J$, then the quotient
$C_W(w)/C_{W_J}(w)$ is isomorphic to $N_J$
(see~\cite{KonvalinkaPfeiffer2010}).

We consider the character $\alpha_J$ of $N(W_J)$, defined, for
$w \in N_W(W_J)$, as
\begin{align*}
  \alpha_J(w) = \det(w |_{\Fix(W_J)}),
\end{align*}
where $\Fix(W_J)$ is the fixed point subspace of $W_J$ in $V$. Note that
$W_J$ is contained in the kernel of $\alpha_J$ and so $\alpha_J(u n) =
\alpha_J (n)$ for $u \in W_J$, $n \in N_J$.

\begin{Lemma}\label{la:alpha-sigma}
  Let $J \subseteq S$.  For $n \in N_J$ denote by $\sigma_J(n)$ the sign
of the permutation induced on $J$ by conjugation with $n$.
Then
\begin{align*}
  \sigma_J(n) = \epsilon(n) \alpha_J(n),
\end{align*}
for all $n \in N_J$.
\end{Lemma}

\begin{proof}
  Denote by $V_J$ the orthogonal complement of $\Fix(W_J)$ in $V$.  Then
  $V_J$  affords  the   reflection  representation  of  the  parabolic
  subgroup $W_J$, and the decomposition  $V = V_J \oplus \Fix(W_J)$ is
  $N_W(W_J)$-stable.  For $n  \in N_J$, the matrix of  $n$ on $V_J$ is
  equivalent to  the permutation matrix  of the conjugation  action of
  $n$ on  $J$ and thus  has determinant $\sigma_J(n)$.  The  matrix of
  $n$  on  $\Fix(W_J)$  has  determinant $\alpha_J(n)$, by definition.
  Consequently,  the  determinant of  $n$  on  $V$  is $\epsilon(n)  =
  \sigma_J(n) \alpha_J(n)$.
\end{proof}

Pfeiffer and R\"ohrle~\cite{PfeifferRoehrle2005} call $W_J$ a \emph{bulky}
parabolic subgroup of $W$ if $N_W(W_J)$ is isomorphic to the direct product
$W_J \times N_J$, or equivalently, if $N_J$ centralizes $W_J$. Notice that
$W_J$ is bulky whenever $W_J$ is a self-normalizing subgroup of $W$. Suppose
$W_J$ is bulky in $W$. Then $\sigma_J(n) = 1$ for all $n \in N_J$.
Consequently, for $u \in W_J$ and $n \in N_J$, we have
\begin{align}\label{eq:tilde-epsilon}
  \epsilon(un) \alpha_J(un) = \epsilon(u).
\end{align}
Thus, the character $\epsilon \alpha_J = \epsilon_J \times 1_{N_J}$ of
$N_W(W_J) = W_J\times N_J$ is the trivial extension of the sign character of
$W_J$.

Here and in the remainder of the paper we denote the restrictions of the trivial and
the sign character of $W$ to a subgroup $U$ of $W$ by $1_U$ and
$\epsilon_U$, respectively, or by $1_J$ and $\epsilon_J$, if $U = W_J$ for
some $J \subseteq S$.  If no confusion can arise, we denote the restrictions
of the characters $1_S$ and $\epsilon_S$ of $W$ to any of its subgroups
simply by $1$ and $\epsilon$, respectively.

Following Bergeron et al.~\cite{BergeronEtAl1992}, we decompose $\Sigma(W)$
into projective indecomposable modules, using a basis of quasi-idempotents,
that naturally arise as follows.  For $L, K \subseteq S$, we define
\[
m_{KL} = 
\begin{cases} 
\Size{\smash{X_K \cap X_L^{\sharp}}}, & \text{if $L \subseteq K$,} \\ 
0,& \text{otherwise.} 
\end{cases}
\]
Then $(m_{KL})_{K,L \subseteq S}$ is an invertible matrix, and consequently,
there is a basis $(e_L)_{L \subseteq S}$ of $\Sigma(W)$ such that
\begin{align*}
  x_K = \sum_{L \subseteq S} m_{KL} e_L
\end{align*}
for $K \subseteq S$.
Define, for $\lambda \in
\Lambda$, elements
\begin{align*}
  e_{\lambda} = \sum_{L \in \lambda} e_L.
\end{align*}
Then $\{e_{\lambda}  : \lambda  \in \Lambda\}$  is a  set of
primitive, pairwise orthogonal idempotents in $\Sigma(W)$.  In particular,
\begin{align*}
  \sum_{\lambda \in \Lambda} e_{\lambda} = 1 \in \kW.
\end{align*}
Thus, if we set
\begin{align*}
  E_{\lambda} = e_{\lambda} \kW,
\end{align*}
then  
\begin{align}\label{eq:regular-decomp}
  \kW = \bigoplus_{\lambda \in \Lambda} E_{\lambda}
\end{align}
is a decomposition of the group algebra into right ideals. We call the right
ideal $E_{[S]}$ the \emph{top component} of $\kW$.

For $\lambda \in \Lambda$, denote by $\Phi_{\lambda}$ the character of the
$W$-module $E_{\lambda}$. Furthermore, for $L \subseteq S$, denote by
$\Phi_L$ the character of the top component of the group algebra $\kW_L$.
Notice that for $\lambda=[L]$, $\Phi_{[L]}$ is a character of $W$ whereas
$\Phi_L$ is a character of $W_L$. If $L=S$, then $W_L=W$ and $\Phi_{[S]} =
\Phi_S$. In general, the characters $\Phi_{[L]}$ and $\Phi_L$ are related in
the following way.

\begin{Proposition}[\protect{\cite[Prop.~3.6(a)]{DouglassPfeiffer2011}}]\label{pro:Phi}
  Let $L \subseteq S$. Then the character $\Phi_L$ of $W_L$ extends to a
  character $\widetilde{\Phi}_L$ of the normalizer $N_W(W_L) = W_L \rtimes
  N_L$ such that
  \begin{align*}
    \Phi_{[L]} &= \Ind_{N_W(W_L)}^W \widetilde{\Phi}_L.
  \end{align*}
\end{Proposition}

\begin{Remark}\label{rem:Phi}
  The argument in the proof of \cite[Prop.~3.6(a)]{DouglassPfeiffer2011}
  shows that if $W_L$ is a bulky parabolic subgroup of $W$, then
  $\widetilde{\Phi}_L$ is the character $\Phi_L \times 1_{N_L}$ of
  $N_W(W_L)=W_L\times N_L$ and so $\Phi_{[L]}= \Ind_{W_L\times N_L}^W (\Phi_L
  \times 1_{N_L})$.
\end{Remark}

\section{The Reflection Arrangement and the Orlik-Solomon Algebra \texorpdfstring{$A(W)$}{A(W)}}
\label{sec:os}

A finite Coxeter group of rank $r$ acts as a reflection group on Euclidean
space $\R^r$.  Here it is convenient to regard this as an action on the
complex space $V_{\C}=\C^r$.  Let
\begin{align*}
  T = \{s^w : s \in S,\, w \in W\}
\end{align*}
be the
set of reflections of~$W$.  For $t \in T$, denote by $H_t$ the reflecting
hyperplane of $t$, i.e., the $1$-eigenspace of $t$.  The set of hyperplanes
$\AA = \{H_t : t \in T\}$ is called the reflection arrangement of~$W$; for
details see~\cite[Ch.~6]{OrlikTerao1992}.  Examples of (the real part of) 
reflection
arrangements in dimension $2$ are shown in Figures~\ref{fig:i2odd}
and~\ref{fig:i2even} below.

\begin{figure}[htbp]
\begin{center}
\begin{tikzpicture}[scale=0.75]
  \draw[very thin,color=gray] (-2.2,-2.2) grid (2.2,2.2);
  \draw[->] (-2.2,0) -- (2.2,0);
  \draw[->] (0,-2.2) -- (0,2.2);
  \draw (0:1) -- (120:1) -- (240:1) -- cycle;
  \foreach \x in {0,120,240} 
    \fill[red] (\x:1) circle (.8mm);
  \foreach \x in {0,120,240} 
    \draw[red,very thick] (\x:-2) -- (\x:2);
  \draw (0:2) node[above] {$_0$};
  \draw (60:2) node[above] {$_1$};
  \draw (120:2) node[above] {$_2$};
\end{tikzpicture}
\quad
\begin{tikzpicture}[scale=0.75]
  \draw[very thin,color=gray] (-2.2,-2.2) grid (2.2,2.2);
  \draw[->] (-2.2,0) -- (2.2,0);
  \draw[->] (0,-2.2) -- (0,2.2);
  \draw (0:1) -- (72:1) -- (144:1) -- (216:1) --(288:1) -- cycle;
  \foreach \x in {0,72,144,216,288} 
    \fill[red] (\x:1) circle (.8mm);
  \foreach \x in {0,72,144,216,288} 
    \draw[red,very thick] (\x:-2) -- (\x:2);
  \draw (0:2) node[above] {$_0$};
  \draw (36:2) node[above] {$_1$};
  \draw (72:2) node[above] {$_2$};
  \draw (108:2) node[above] {$_3$};
  \draw (144:2) node[above] {$_4$};
\end{tikzpicture}
\quad
\begin{tikzpicture}[scale=0.75]
  \draw[very thin,color=gray] (-2.2,-2.2) grid (2.2,2.2);
  \draw[->] (-2.2,0) -- (2.2,0);
  \draw[->] (0,-2.2) -- (0,2.2);
  \draw (0:1) -- (51:1) -- (103:1) -- (154:1) --(206:1) --(257:1) --(309:1) -- cycle;
  \foreach \x in {0,51,103,154,206,257,309} 
    \fill[red] (\x:1) circle (.8mm);
  \foreach \x in {0,51,103,154,206,257,309} 
    \draw[red,very thick] (\x:-2) -- (\x:2);
  \draw (0:2) node[above] {$_0$};
  \draw (26:2) node[above] {$_1$};
  \draw (51:2) node[above] {$_2$};
  \draw (77:2) node[above] {$_3$};
  \draw (103:2) node[above] {$_4$};
  \draw (129:2) node[above] {$_5$};
  \draw (154:2) node[above] {$_6$};
\end{tikzpicture}
\quad
\begin{tikzpicture}[scale=0.75]
  \draw[very thin,color=gray] (-2.2,-2.2) grid (2.2,2.2);
  \draw[->] (-2.2,0) -- (2.2,0);
  \draw[->] (0,-2.2) -- (0,2.2);
  \draw (0:1) -- (40:1) -- (80:1) -- (120:1) --(160:1) --(200:1) --(240:1) --(280:1) --(320:1) -- cycle;
  \foreach \x in {0,40,80,120,160,200,240,280,320} 
    \fill[red] (\x:1) circle (.8mm);
  \foreach \x in {0,40,80,120,160,200,240,280,320} 
    \draw[red,very thick] (\x:-2) -- (\x:2);
  \draw (0:2) node[above] {$_0$};
  \draw (20:2) node[above] {$_1$};
  \draw (40:2) node[above] {$_2$};
  \draw (60:2) node[above] {$_3$};
  \draw (80:2) node[above] {$_4$};
  \draw (100:2) node[above] {$_5$};
  \draw (120:2) node[above] {$_6$};
  \draw (140:2) node[above] {$_7$};
  \draw (160:2) node[above] {$_8$};
\end{tikzpicture}
\end{center}
  \caption{Hyperplane Arrangements of Type $I_2(m)$, $m = 3,5,7,9$.}
  \label{fig:i2odd}
\end{figure}
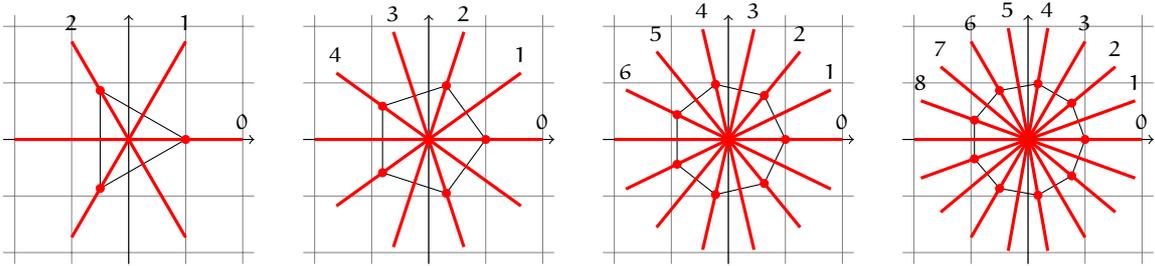

\begin{figure}[htbp]
\begin{center}
\begin{tikzpicture}[scale=0.75]
  \draw[very thin,color=gray] (-2.2,-2.2) grid (2.2,2.2);
  \draw[->] (-2.2,0) -- (2.2,0);
  \draw (0,-2.2) -- (0,2.2);
  \draw (0:1) -- (90:1) --(180:1) -- (270:1) -- cycle;
  \foreach \x in {0,90,180,270} 
    \fill[red] (\x:1) circle (.8mm);
  \foreach \x in {45,135} 
    \draw[blue,very thick] (\x:-2) -- (\x:2);
  \foreach \x in {0,90} 
    \draw[red,very thick] (\x:-2) -- (\x:2);
  \draw (0:2) node[above] {$_0$};
  \draw (45:2) node[above] {$_1$};
  \draw (90:2) node[above] {$_2$};
  \draw (135:2) node[above] {$_3$};
\end{tikzpicture}
\quad
\begin{tikzpicture}[scale=0.75]
  \draw[very thin,color=gray] (-2.2,-2.2) grid (2.2,2.2);
  \draw[->] (-2.2,0) -- (2.2,0);
  \draw (0,-2.2) -- (0,2.2);
  \draw (0:1) -- (60:1) --(120:1) -- (180:1) --(240:1) -- (300:1) -- cycle;
  \foreach \x in {0,60,120,180,240,300} 
    \fill[red] (\x:1) circle (.8mm);
  \foreach \x in {30,90,150} 
    \draw[blue,very thick] (\x:-2) -- (\x:2);
  \foreach \x in {0,60,120} 
    \draw[red,very thick] (\x:-2) -- (\x:2);
  \draw (0:2) node[above] {$_0$};
  \draw (30:2) node[above] {$_1$};
  \draw (60:2) node[above] {$_2$};
  \draw (90:2) node[above] {$_3$};
  \draw (120:2) node[above] {$_4$};
  \draw (150:2) node[above] {$_5$};
\end{tikzpicture}
\quad
\begin{tikzpicture}[scale=0.75]
  \draw[very thin,color=gray] (-2.2,-2.2) grid (2.2,2.2);
  \draw[->] (-2.2,0) -- (2.2,0);
  \draw (0,-2.2) -- (0,2.2);
  \draw (0:1) -- (45:1) -- (90:1) -- (135:1) --(180:1) --(225:1) --(270:1) --(315:1) -- cycle;
  \foreach \x in {0,45,90,135,180,225,270,315} 
    \fill[red] (\x:1) circle (.8mm);
  \foreach \x in {23,68,113,158}
    \draw[blue,very thick] (\x:-2) -- (\x:2);
  \foreach \x in {0,45,90,135}
    \draw[red,very thick] (\x:-2) -- (\x:2);
  \draw (0:2) node[above] {$_0$};
  \draw (23:2) node[above] {$_1$};
  \draw (45:2) node[above] {$_2$};
  \draw (68:2) node[above] {$_3$};
  \draw (90:2) node[above] {$_4$};
  \draw (113:2) node[above] {$_5$};
  \draw (135:2) node[above] {$_6$};
  \draw (158:2) node[above] {$_7$};
\end{tikzpicture}
\quad
\begin{tikzpicture}[scale=0.75]
  \draw[very thin,color=gray] (-2.2,-2.2) grid (2.2,2.2);
  \draw[->] (-2.2,0) -- (2.2,0);
  \draw (0,-2.2) -- (0,2.2);
  \draw (0:1) -- (36:1) -- (72:1) -- (108:1) --(144:1) --(180:1) -- (216:1) -- (252:1) -- (288:1) -- (324:1) -- cycle;
  \foreach \x in {0,36,72,108,144,180,216,252,288,324} 
    \fill[red] (\x:1) circle (.8mm);
  \foreach \x in {18,54,90,126,162} 
    \draw[blue,very thick] (\x:-2) -- (\x:2);
  \foreach \x in {0,36,72,108,144} 
    \draw[red,very thick] (\x:-2) -- (\x:2);
  \draw (0:2) node[above] {$_0$};
  \draw (18:2) node[above] {$_1$};
  \draw (36:2) node[above] {$_2$};
  \draw (54:2) node[above] {$_3$};
  \draw (72:2) node[above] {$_4$};
  \draw (90:2) node[above] {$_5$};
  \draw (108:2) node[above] {$_6$};
  \draw (126:2) node[above] {$_7$};
  \draw (144:2) node[above] {$_8$};
  \draw (162:2) node[above] {$_9$};
\end{tikzpicture}
\end{center}
  \caption{Hyperplane Arrangements of Type $I_2(m)$, $m = 4,6,8,10$.}
  \label{fig:i2even}
\end{figure}
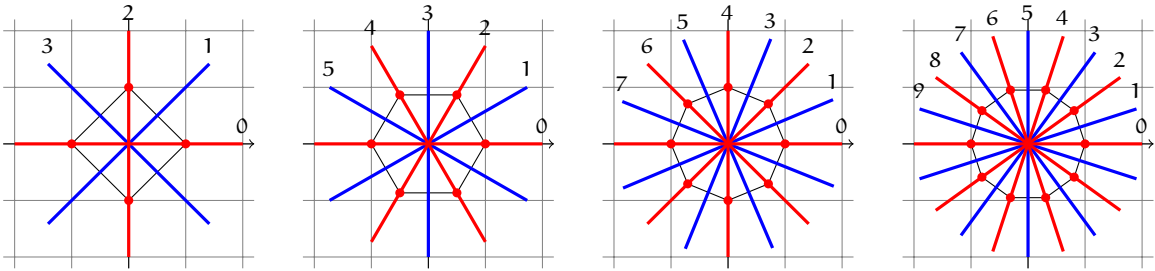

The \emph{lattice} of $\AA$ is the set of all possible intersections
of hyperplanes
\begin{align*}
  L(\AA) = \{H_{t_1} \cap \dots \cap H_{t_p} : t_1, \dots, t_p \in T\}.
\end{align*}
For $X \in L(\AA)$, the pointwise stabilizer
\begin{align*}
  W_X = \{w \in W : x.w = x \text{ for all } x \in X\}
\end{align*}
is a parabolic subgroup of $W$.  We define the \emph{shape} $\sh(X)$ of $X$
to be the shape of $W_X$, i.e., $\sh(X) = [L] \in \Lambda$ if $W_X$ is
conjugate to $W_L$ in~$W$ for some $L \subseteq S$. 
The group $W$ acts on $T$ by conjugation and the
$W$-action on $T$ induces actions of $W$ on $\AA$ and $L(\AA)$.  Orlik and
Solomon~\cite{OrlikSolomon1983} have shown that the normalizer of $W_X$ in
$W$ is the setwise stabilizer of $X$ in $W$, that is
\begin{align*}
  N_W(W_X) = \{w \in W: X.w = X\}.
\end{align*}
Consequently, the orbits of $W$ on the lattice $L(\AA)$ are parametrized
by the shapes of $W$.  We denote by $\alpha_X \colon N_W(W_X) \to \C$ the
linear character of $N_W(W_X)$ defined by
\begin{align*}
  \alpha_X(w) = \det(w|_X)
\end{align*}
for  $w \in  N_W(W_X)$.  Then,  for  $w \in  W$, we  have $\alpha_w  =
\alpha_X$, where  $X =  \Fix(w)$, the fixed  point subspace of  $w$ in
$V$.  Moreover,  for $L \subseteq  S$, we have $\alpha_L  = \alpha_X$,
where $X = \Fix(W_L)$.

The Orlik-Solomon algebra of $W$ is the associative $\C$-algebra
$A(W)$, generated as an algebra by elements $a_t$, $t \in T$,
subject to the relations
\begin{align*}
  a_t a_{t'} = - a_{t'} a_t
\end{align*}
for all $t, t' \in T$,
and 
\begin{align*}
  \sum_{i=1}^p (-1)^i a_{t_1} \dotsm a_{t_{i-1}} \widehat{a_{t_i}} a_{t_{i+1}} \dotsm a_{t_p} = 0,
\end{align*}
where the hat denotes omission, whenever $\{H_{t_1}, \dots, H_{t_p}\}$
is linearly dependent.  The action of $W$ on the hyperplanes extends to an
action on $A(W)$ via
\begin{align*}
  a_t.w = a_{t^w}
\end{align*}
for $t \in T$, $w \in W$.  The algebra $A(W)$ is a skew-commutative, graded
algebra
\begin{align*}
  A(W) = \bigoplus_{p \geq 0} A^p,
\end{align*}
where the degree $p$ subspace $A^p$ is spanned by those monomials $a_{t_1}
\dotsm a_{t_p}$ in $A(W)$ with $\dim H_{t_1} \cap \dots \cap H_{t_p} = r-p$.
Clearly, $A^p = 0$ for $p > r$.  We call $A^r$ the \emph{top component} of
$A(W)$. We need a refinement of this decomposition, due to Brieskorn
\cite{Brieskorn1973}. For a subspace $X \in L(\AA)$ of codimension $p$,
define a subspace
\begin{align*}
  A_X = \Span{a_{t_1} \dotsm a_{t_p} : H_{t_1} \cap \dots \cap H_{t_p} = X}
\end{align*}
of $A(W)$. Then $A_{\{0\}} = A^r$ is the top component of $A(W)$.  Note that
$A_X$ is an embedding of the top component of $A(W_X)$ into $A(W)$.  For $w
\in W$, we have $A_X.w = A_{X.w}$ and so $A_X$ is an $N_W(W_X)$-stable
subspace.

We have
\begin{align*}
  A(W) = \bigoplus_{X \in L(\AA)} A_X
\end{align*}
and if we set
\begin{align*}
  A_{\lambda} = \bigoplus_{\sh(X) = \lambda} A_X, 
\end{align*}
for $\lambda \in \Lambda$, 
then
\begin{align*}
  A(W) = \bigoplus_{\lambda \in \Lambda} A_{\lambda}
\end{align*}
is a decomposition of $A(W)$ into $W$-modules $A_{\lambda}$.  Note that
$A_{[S]} = A_{\{0\}}$ is the top component of $A(W)$. 

For $\lambda \in \Lambda$, denote by $\Psi_{\lambda}$ the character of the
component $A_{\lambda}$ of the Orlik-Solomon algebra $A(W)$. Furthermore,
for $L \subseteq S$, denote by $\Psi_L$ the character of the top component
of the Orlik-Solomon algebra $A(W_L)$ of the parabolic subgroup $W_L$
of~$W$. Notice that for $\lambda=[L]$, $\Psi_{[L]}$ is a character of $W$
whereas $\Psi_L$ is a character of $W_L$. If $L=S$, then $\Psi_{[S]} =
\Psi_S$. In general, the characters $\Psi_{[L]}$ and $\Psi_L$ are related in
the following way, analogous to
Proposition~\ref{pro:Phi}.

\begin{Proposition}[\protect{\cite[\S2]{LehrerSolomon1986}}]\label{pro:Psi}
  Let $L \subseteq S$. Then the character $\Psi_L$ of $W_L$ extends to a
  character $\widetilde{\Psi}_L$ of the normalizer $N_W(W_L) = W_L \rtimes
  N_L$ such that
  \begin{align*}
    \Psi_{[L]} &= \Ind_{N_W(W_L)}^W \widetilde{\Psi}_L.
  \end{align*}
\end{Proposition}

\begin{Remark}\label{rem:Psi}
  Suppose that $W_L$ is a bulky parabolic subgroup of $W$ and set $X=
  \Fix(W_L)$. If $\codim\, X=p$ and $t_1, \dots, t_p$ are in $T$ with
  $X=H_{t_1} \cap \dots \cap H_{t_p}$, then $t_1, \dots, t_p$ are in $W_L$
  and so, since $N_L$ centralizes $W_L$, we have $a_{t_1} \dotsm a_{t_p}.n=
  a_{t_1^n} \dotsm a_{t_p^n}= a_{t_1} \dotsm a_{t_p}$, for $n \in N_L$. Thus,
  $\widetilde{\Psi}_L$ is the character $\Psi_L \times 1_{N_L}$ of
  $N_W(W_L)=W_L\times N_L$ and so $\Psi_{[L]}= \Ind_{W_L\times N_L}^W (\Psi_L
  \times 1_{N_L})$.
\end{Remark}

\section{The Inductive Strategy}
\label{sec:claim}

Before stating our relative Conjecture ~\ref{conj:c}, we briefly review the
proof of Conjecture~2.1 in~\cite{DouglassPfeiffer2011} and describe how it
leads to a proof of Conjecture~\ref{conj:a}.
We first showed that the characters of the top components of $\kW$ and
$A(W)$ are related as described in the following conjecture which makes
sense for any finite Coxeter group.
To this end, let $\CC$ be the set of cuspidal conjugacy classes of $W$ and, for
$L\subseteq S$, let $\CC_L$ denote the set of cuspidal conjugacy classes in
$W_L$. For a class $C$ in $\CC$ or $\CC_L$, we denote by $w_C\in C$ a fixed
representative.

\begin{Conjecture}\label{conj:b}
  For each class $C \in \CC$, there exist linear characters $\varphi_{w_C}$
  and $\psi_{w_C}$ of the centralizer $C_W(w_C)$ such that the following
  hold:
  \begin{enumerate}
  \item $\Phi_S = \sum_{C \in \CC} \Ind_{C_W(w_C)}^W \varphi_{w_C}$;
  \item $\Psi_S = \sum_{C \in \CC} \Ind_{C_W(w_C)}^W \psi_{w_C}$;
  \item $\psi_{w_C} = \varphi_{w_C} \epsilon$ for all $C\in \CC$.
  \end{enumerate}
\end{Conjecture}

\begin{Remark}\label{rem:comb}
  If it is known that $\Psi_S= \Phi_S\epsilon_S$, then choosing
  $\psi_{w_C}$ or $\varphi_{w_C}$ in such a way that $\psi_{w_C} =
  \varphi_{w_C} \epsilon$, we have that part (iii) in the above
  Conjecture~\ref{conj:b} holds and that (i) and (ii) are equivalent
  statements.
\end{Remark}

When $W$ is a symmetric group, every parabolic subgroup $W_L$ of $W$ is a
product of symmetric groups and so Conjecture~\ref{conj:b} 
holds for the group $W_L$. Thus,
for $w_C \in C \in \CC_L$, we obtained linear characters $\varphi_{w_C}$ and
$\psi_{w_C}$ of $C_{W_L}(w_C)$ such that the characters $\Phi_L$ and
$\Psi_L$ of $W_L$ decompose as
\begin{equation*}
  \label{eq:1}
  \Phi_L = \sum_{C \in \CC_L} \Ind_{C_{W_L}(w_C)}^{W_L}
  \varphi_{w_C} \quad\text{and}\quad \Psi_L = \sum_{C
    \in \CC_L} \Ind_{C_{W_L}(w_C)}^{W_L} \psi_{w_C}.
\end{equation*}
We know from Propositions \ref{pro:Phi} and \ref{pro:Psi} that $\Phi_L$ and
$\Psi_L$ extend to characters $\widetilde{\Phi}_L$ and $\widetilde{\Psi}_L$
of $N_W(W_L)$. The next step in~\cite{DouglassPfeiffer2011} 
was to show that each $\varphi_{w_C}$ and
$\psi_{w_C}$ extend to characters $\widetilde{\varphi}_{w_C}$ and
$\widetilde{\psi}_{w_C}$ of $C_W(w_C)$ in such a way that
\begin{equation}\label{eq:4.2}
  \widetilde{\Phi}_L = \sum_{C \in \CC_L} \Ind_{C_W(w_C)}^{N_W(W_L)}
  \widetilde{\varphi}_{w_C} \quad\text{and}\quad  \widetilde{\Psi}_L =
  \sum_{C \in \CC_L} \Ind_{C_W(w_C)}^{N_W(W_L)} \widetilde{\psi}_{w_C},
\end{equation}
and moreover that $\widetilde{\psi}_{w_C}= \widetilde{\varphi}_{w_C}
\epsilon_S \alpha_L$ for all $C\in \CC_L$. Finally, we applied
$\Ind_{N_W(W_L)}^W$ to (\ref{eq:4.2}) and summed over the set of shapes
$[L]\in \Lambda$. Conjecture~\ref{conj:a} then follows immediately by
transitivity of induction.

Motivated by (\ref{eq:4.2}) we make the following general conjecture.

\begin{Conjecture}\label{conj:c}
  Let $L \subseteq S$. Then, for each $C \in \CC_L$, there exist linear
  characters $\widetilde{\varphi}_{w_C}$ and $\widetilde{\psi}_{w_C}$ of
  $C_W(w_C)$ such that the following hold:
  \begin{enumerate}
  \item $\widetilde{\Phi}_L = \sum_{C \in \CC_L} \Ind_{C_W(w_C)}^{N_W(W_L)} \widetilde{\varphi}_{w_C}$;
  \item $\widetilde{\Psi}_L = \sum_{C \in \CC_L} \Ind_{C_W(w_C)}^{N_W(W_L)}
      \widetilde{\psi}_{w_C}$;
  \item $\widetilde{\psi}_{w_C}= \widetilde{\varphi}_{w_C} \epsilon_S
    \alpha_L$ for all $C\in \CC_L$. 
  \end{enumerate}
\end{Conjecture}

\begin{Remark}\label{rem:comc}
  If it is known that $\widetilde{\Psi}_L= \widetilde{\Phi}_L \epsilon_S
  \alpha_L$, then choosing $\widetilde{\psi}_{w_C}$ or
  $\widetilde{\varphi}_{w_C}$ in such a way that $\widetilde{\psi}_{w_C} =
  \widetilde{\varphi}_{w_C} \epsilon_S \alpha_L$, we have that part (iii) 
  in the above Conjecture~\ref{conj:c} holds
  and that (i) and (ii) are equivalent statements.
\end{Remark}

Conjecture \ref{conj:b} is known to hold in the following cases:
\begin{enumerate}\renewcommand{\labelenumi}{\arabic{enumi}.}
\item $W$ of type $A$ (see \cite[Thm.~4.1]{DouglassPfeiffer2011});
\item $W$ has rank $2$ or less (see Lemmas~\ref{la:a0} and~\ref{la:a1}, Theorem~\ref{thm:BforDihedral}).
\end{enumerate}

Conjecture \ref{conj:c} is known to hold in the following cases:
\begin{enumerate}\renewcommand{\labelenumi}{\arabic{enumi}.}
\item $W$ of type $A$; all $L$ (see \cite[Thm.~5.2]{DouglassPfeiffer2011});
\item $W$ arbitrary; $W_L$ is
bulky and satisfies Conjecture \ref{conj:b} (by Theorem \ref{thm:bulky});
\item $W$ arbitrary; $\Size{L} \leq 2$ (see Corollary~\ref{cor:0,1}, Theorem~\ref{thm:CforDihedral}).
\end{enumerate}

If Conjecture ~\ref{conj:c} holds for all $L\subseteq S$, then Conjecture
~\ref{conj:a} is true for $W$.

\begin{Theorem}\label{thm:c=>a}
  Suppose that Conjecture ~\ref{conj:c} holds for all subsets $L\subseteq S$.
  Then for each $w$ in a set $\RR$ of representatives of the conjugacy
  classes of $W$, there are linear characters $\widetilde{\varphi}_w$ and
  $\widetilde{\psi}_w$ of $C_W(w)$ such that
  \begin{enumerate}
  \item the regular character of $W$ is given by $\rho_W = \sum_{w \in \RR}
    \Ind_{C_W(w)}^W \widetilde{\varphi}_{w}$,
  \item the Orlik-Solomon character of $W$ is given by $\omega_W = \sum_{w
      \in \RR} \Ind_{C_W(w)}^W \widetilde{\psi}_{w}$, and
  \item $\widetilde{\psi}_w = \widetilde{\varphi}_w \epsilon \alpha_w$ for all $w \in \RR$.
  \end{enumerate}
\end{Theorem}

\begin{proof}
  For $L\subseteq S$, let $\RR_L$ be a set of minimal length representatives
  of the classes $\CC_L$. For a class $C \in \CC_L$, denote by $w_C \in
  \RR_L$ its representative. Let $\LL$ be a set of representatives of
  shapes, so $\Lambda = \{\, [L]\mid L \in \LL\,\}$. Then,
  by~\cite[Thm.~3.2.12]{geckpfeiffer2000}, we may assume without loss that
  \begin{align*}
    \RR = \coprod_{L \in \LL} \RR_L = \{w_C : C \in \CC_L,\, L \in \LL\}.
  \end{align*}
  Then the equality in (iii) holds. By~(\ref{eq:regular-decomp}) and
  Proposition~\ref{pro:Phi}, we have
  \begin{align*}
    \rho_W &= \sum_{\lambda \in \Lambda} \Phi_{\lambda} = \sum_{L \in \LL}
    \Ind_{N_W(W_L)}^W \widetilde{\Phi}_L = \sum_{L \in \LL} \sum_{C \in
      \CC_L} \Ind_{C_W(w_C)}^W \widetilde{\varphi}_{w_C},
  \end{align*}
  as desired. The formula for $\omega_W$ follows in the same way.
\end{proof}

Notice that in the case when $L=S$, Conjecture~\ref{conj:c} for $L\subseteq
S$ is simply a restatement of Conjecture~\ref{conj:b}. In general,
Conjecture~\ref{conj:c} for $L\subseteq S$ implies the validity of
Conjecture~\ref{conj:b} for the group $W_L$, as follows.

\begin{Proposition}\label{pro:c=>b}
Suppose that Conjecture~\ref{conj:c} holds for a subset $L\subseteq S$.
Then the restrictions
\[
\varphi_{w_C} = \Res^{C_W(w_C)}_{C_{W_L}(w_C)} \widetilde{\varphi}_{w_C}
\quad\text{and}\quad \psi_{w_C} = \Res^{C_W(w_C)}_{C_{W_L}(w_C)}
\widetilde{\psi}_{w_C}
\]
are linear characters that satisfy Conjecture~\ref{conj:b} for $W_L$.
\end{Proposition}

\begin{proof}
  By Mackey's theorem, we have
\begin{align*}
  \Res^{N_W(W_L)}_{W_L} \Ind_{C_W(w_C)}^{N_W(W_L)} \widetilde{\varphi}_{w_C}
  &= \Ind_{C_{W_L}(w_C)}^{W_L} \Res^{C_W(w_C)}_{C_{W_L}(w_C)}
  \widetilde{\varphi}_{w_C},
\end{align*}
since $N_W(W_L) = W_L C_W(w_C)$ (see~\cite{KonvalinkaPfeiffer2010}), and
therefore,
\begin{align*}
  \Phi_L &= \Res^{N_W(W_L)}_{W_L} \widetilde{\Phi}_L \\
  &= \sum_{C \in \CC_L} \Res^{N_W(W_L)}_{W_L} \Ind_{C_W(w_C)}^{N_W(W_L)}
  \widetilde{\varphi}_{w_C} \\
  &= \sum_{C \in \CC_L} \Ind_{C_{W_L}(w_C)}^{W_L}
  \Res^{C_W(w_C)}_{C_{W_L}(w_C)} \widetilde{\varphi}_{w_C}\\
  & = \sum_{C \in \CC_L} \Ind_{C_{W_L}(w_C)}^{W_L} \varphi_{w_C}.
\end{align*}
The formula for $\Psi_L$ follows in the same way. The conclusion that
$\psi_{w_C} = \varphi_{w_C} \epsilon$ for $C\in \CC_L$ is easily seen to
hold.
\end{proof}

\begin{Remark}\label{rem:proveC}
Although Conjecture ~\ref{conj:b} for $W_L$ formally follows from
Conjecture~\ref{conj:c}, as in ~\cite{DouglassPfeiffer2011}, the characters
$\widetilde{\varphi}_{w_C}$ and $\widetilde{\psi}_{w_C}$ of $C_W(w_C)$ arise
in practice as extensions of characters $\varphi_{w_C}$ and $\psi_{w_C}$ of
$C_{W_L}(w_c)$ that satisfy Conjecture~\ref{conj:b} for $W_L$. In
particular, if Conjecture~\ref{conj:b} is known to hold for $W_L$, then
using Remark~\ref{rem:comc}, to prove Conjecture~\ref{conj:c} for
$L\subseteq S$, it suffices to prove that each $\varphi_{w_C}$ extends to
$C_W(w_C)$ in such a way that Conjecture~\ref{conj:c} (i) holds and that
$\widetilde{\Psi}_L= \widetilde{\Phi}_L \epsilon_S \alpha_L$.
\end{Remark}

When $L \subseteq S$ is such that $W_L$ is a self-normalizing subgroup
of $W$ (e.g., if $L = S$), then $N_L$ is the trivial group and Conjecture
\ref{conj:b} for the group $W_L$ vacuously implies Conjecture \ref{conj:c}
for the subset $L$ in this case. More generally, whenever the
complement $N_L$ centralizes $W_L$, i.e., when $W_L$ is bulky in $W$,
Conjecture~\ref{conj:b} for $W_L$ implies Conjecture~\ref{conj:c} for
$L\subseteq S$, as follows.

\begin{Theorem}\label{thm:bulky}
  Let $L \subseteq S$. Suppose that Conjecture~\ref{conj:b} holds for the
  group $W_L$ and that $W_L$ is a bulky parabolic subgroup of $W$. Then
  Conjecture~\ref{conj:c} holds with $\widetilde{\varphi}_{w_C} =
  \varphi_{w_C} \times 1_{N_L}$ and $\widetilde{\psi}_{w_C} = \psi_{w_C}
  \times 1_{N_L}$ for each cuspidal class $C$ of $W_L$.
\end{Theorem}

\begin{proof}
  As observed in the remark above, it suffices to show that each
  $\varphi_{w_C}$ extends to $C_W(w_C)$ in such a way that
  Conjecture~\ref{conj:c} (i) holds and that $\widetilde{\Psi}_L=
  \widetilde{\Phi}_L \epsilon_S \alpha_L$.

  Because $N_L$ centralizes $W_L$, we have that the centralizer $C_W(w_C)$ is
  the direct product of $C_{W_L}(w_C)$, and $N_L$ and so
  $\widetilde{\varphi}_{w_C}$ is indeed a linear character of $C_W(w_C)$
  that extends $\varphi_{w_C}$.  Thanks to Remark~\ref{rem:Phi},
  $\widetilde{\Phi}_L = \Phi_L \times 1_{N_L}$.  Thus, by
  Conjecture~\ref{conj:b} (i) we have,
  \begin{align*}
    \widetilde{\Phi}_L &= \Phi_L \times 1_{N_L} = \sum_{C \in \CC_L}
    \Ind_{C_{W_L}(w_C)}^{W_L} \varphi_{w_C} \times 1_{N_L} \\
    &= \sum_{C \in \CC_L} \Ind_{C_{W_L}(w_C)\times N_L}^{W_L\times N_L}
    (\varphi_{w_C} \times 1_{N_L}) \\
    &= \sum_{C \in \CC_L} \Ind_{C_W(w_C)}^{N_W(W_L)}
    \widetilde{\varphi}_{w_C}.
  \end{align*}
  Hence Conjecture~\ref{conj:c} (i) holds.

  By Remark~\ref{rem:Psi}, Conjecture~\ref{conj:b} (iii),
  Remark~\ref{rem:Phi}, and Lemma~\ref{la:alpha-sigma}, we have
  \begin{align*}
    \widetilde{\Psi}_L = \Psi_L \times 1_{N_L} = \Phi_L \epsilon_L \times
    1_{N_L} \sigma_L = (\Phi_L \times 1_{N_L}) \epsilon \alpha_L =
    \widetilde{\Phi}_L \epsilon_S \alpha_L,
  \end{align*}
  using the fact that $W_L \subseteq \ker \alpha_L$, whence we are done.
\end{proof}

Combining Theorem \ref{thm:bulky} with the results
in~\cite{DouglassPfeiffer2011}, we see that if $W_L$ is a product of Coxeter
groups of type $A$ and is a bulky parabolic subgroup of $W$, then Conjecture
~\ref{conj:c} holds for $L\subseteq S$. For example, if $W_L$ is of type
$A_1\times A_3$ and $W$ is of type $E_6$, then the characters
$\varphi_{w_C}$ and $\psi_{w_C}$ constructed in ~\cite{DouglassPfeiffer2011}
satisfy Conjecture ~\ref{conj:b} and so, by Theorem \ref{thm:bulky}, they
extend to $C_W(w_C)$ and Conjecture ~\ref{conj:c} holds. Note however, that
the property of being a bulky parabolic subgroup depends in a fundamental way
on the embedding of $W_L$ in $W$. If $W_L$ is of type $A_1\times A_3$ and
$W$ is of type $E_7$, then $W_L$ is not bulky and Theorem \ref{thm:bulky}
cannot be applied.

\section{Conjectures \ref{conj:a}, \ref{conj:b} and \ref{conj:c} for Coxeter Groups of Rank up to \texorpdfstring{$2$}{2}}
\label{sec:dihedral}

In this section we show that Conjecture ~\ref{conj:c} holds for $L\subseteq
S$ for any $S$ as long as $\Size{L} \leq 2$. Note that because the type of
the ambient Coxeter group $W$ is arbitrary, even for types $A_1 \times A_1$
and $A_2$ Conjecture ~\ref{conj:c} is a stronger statement than is proved in
\cite{DouglassPfeiffer2011} for such parabolic subgroups. The strategy we
use is to first prove that Conjecture ~\ref{conj:b} holds for $W$ when the
rank of $W$ is at most $2$ and then use the procedure outlined in
Remark~\ref{rem:proveC}.  Combining Conjecture ~\ref{conj:c} with Theorem
\ref{thm:bulky} we conclude that Conjectures ~\ref{conj:a}, ~\ref{conj:b}, and
~\ref{conj:c} all hold in case the rank of $W$ is at most two.

The top components of Coxeter groups of rank $0$ or $1$ almost trivially
satisfy Conjecture~\ref{conj:b}. For later reference, we record this
explicitly in the following lemmas.

\begin{Lemma}\label{la:a0}
  The top component characters of $W_{\emptyset}$ are
  $\Phi_{\emptyset} = 1_{\emptyset}$ and $\Psi_{\emptyset} =
  1_{\emptyset}$.  Moreover, $W_{\emptyset}$ satisfies
  Conjecture~\ref{conj:b} with $\varphi_1 = 1_{\emptyset}$ and $\psi_1 =
  1_{\emptyset}$.
\end{Lemma}

\begin{Lemma}\label{la:a1}
  Suppose $W$ is a Coxeter group of rank $1$, generated by $S =
  \{s\}$.  Then the top component characters of $W$ are $\Phi_S =
  \epsilon_S$ and $\Psi_S = 1_S$.  Moreover, $W$ satisfies
  Conjecture~\ref{conj:b} with $\varphi_s = \epsilon_S$ and $\psi_s =
  1_S$.
\end{Lemma}

\begin{proof}
  In this case, the non-trivial conjugacy class $\{s\}$ is the unique
  cuspidal conjugacy class in $W$. From the definitions we have $e_{[S]}=
  e_S = \frac12(1 - s)$ and it follows that $W$ acts on the top component
  $E_{[S]} = e_{[S]} \kW$ with character $\Phi_{[S]} = \epsilon_S$.
  Moreover, $W$ acts trivially on the basis $\{a_s\}$ of the top component
  $A_{[S]}$ of $A(W)$, which therefore affords the trivial character. Thus,
  $\Psi_{[S]} = 1_S$ and so $\Phi_{[S]}=\Psi_{[S]} \epsilon_S$. Set
  $\varphi_s = \epsilon_S$ and $\psi_s = 1_S$. Then $\varphi_s$ and $\psi_s$
  obviously satisfy the conclusions of Conjecture~\ref{conj:b}.
\end{proof}

In any finite Coxeter group $W$, parabolic subgroups of rank $0$ and $1$ are
always bulky. We may thus conclude from Lemmas \ref{la:a0} and \ref{la:a1}
and Theorem \ref{thm:bulky} that Conjecture ~\ref{conj:c} holds for
$L\subseteq S$ with $|L|\leq 1$.

\begin{Corollary}\label{cor:0,1}
  Suppose that $L \subseteq S$ has size $\Size{L} \leq 1$.  Then
  Conjecture~\ref{conj:c} holds.
\end{Corollary}

As a consequence of the corollary, $W$ acts trivially on both the component
$E_{[\emptyset]}$ of the group algebra $\kW$ (with character
$\Phi_{[\emptyset]} = \widetilde{\Phi}_{\emptyset} = 1_S$) and the component
$A_{[\emptyset]}$ of the Orlik-Solomon algebra $A(W)$ (with character
$\Psi_{[\emptyset]} = \widetilde{\Psi}_{\emptyset} = 1_S$), as one can
easily establish directly.

Moreover, the degree $1$ component of $A(W)$ is a direct sum of transitive
permutation modules, one for each conjugacy class of reflections of $W$.
This agrees with the description of the degree $1$ component of $A(W)$ as
the permutation representation of $W$ on its reflections, that can easily be
obtained directly.

Next we consider the case when $W$ has rank $2$. Until further notice, we
assume that
\begin{align*}
  W = \Span{s, t : s^2 = t^2 =  (st)^m = 1}. 
\end{align*}
Then $W$ is a Coxeter group of rank two and is of type $A_1 \times A_1$, or
$I_2(m)$ for $m \geq 3$, with Coxeter generators $S = \{s, t\}$. For
convenience, we regard type $A_1 \times A_1$ as type $I_2(2)$, noting that
the general results of this section remain true for $m = 2$.

To prove Conjecture~\ref{conj:b} for $W$, we first compute the character
$\Phi_S$ of the top component $E_{[S]}$ of the group algebra $\kW$, and
verify that it is a sum of induced linear characters.  Then we compute the
character $\Psi_S$ of the top component $A_{[S]}$ of the Orlik-Solomon
algebra $A(W)$ and verify that $\Psi_S = \Phi_S \epsilon_S$. Conjecture
~\ref{conj:b} then follows as observed in Remark \ref{rem:comb}.

As usual, denote by $w_0$ the longest element of $W$.  Furthermore, we
denote 
\begin{align*}
  \Av(U) = \frac1{\Size{U}} \sum_{u \in U} u
\end{align*}
for a subgroup $U$ of
$W$.  Recall that $\Av(U) u = \Av(U)$ for all $u \in U$ and that $\Av(U) \C
W$ is the permutation module of $W$ on the cosets of~$U$.

\begin{Lemma}\label{la:eS} 
$e_S = \Av(\Span{w_0}) - \Av(W)$.
\end{Lemma}

\begin{proof}
  By Solomon's theorem~\cite{Solomon1976}, the elements
\begin{align*}
x_{\emptyset} & = 1 + s + t + st + ts + \dots + w_0, &
  x_s &= 1 + t + st + tst + \dots + w_0s, \\
  x_{st} &= 1, &
  x_t &= 1 + s + ts + sts + \dots + w_0t
\end{align*}
form a basis of the descent algebra $\Sigma(W)$.
Note that $x_t + x_s = x_{\emptyset} + 1 - w_0$.

For $L \subseteq K \subseteq S$, 
the numbers $m_{KL} = \Size{\smash{X_K \cap X_L^{\sharp}}}$ are easily determined
as
\begin{align*}
(m_{KL})_{K, L \subseteq S} &=
  \left[
    \begin{array}{cccc}
      2m & . & . & . \\
      m & 2 & . & . \\
      m & . & 2 & . \\
      1 & 1 & 1 & 1
    \end{array}
\right],&
(m_{KL})^{-1} &=
  \left[
    \begin{array}{cccc}
      \tfrac1{2m} & . & . & . \\
      -\tfrac14 & \tfrac12 & . & . \\
      -\tfrac14 & . & \tfrac12 & . \\
      \tfrac{m-1}{2m} & -\tfrac12 & -\tfrac12 & 1
    \end{array}
\right].&
\end{align*}
Hence the idempotents $e_L$ are
(cf.~\cite{BergeronEtAl1992}) 
\begin{align*}
  e_{\emptyset} &= \tfrac1{2m} x_{\emptyset}, &
  e_s &= \tfrac12 x_s - \tfrac14 x_{\emptyset}, \\
  e_{st} &= 1 - \tfrac12 x_s - \tfrac12 x_t + \tfrac{m-1}{2m} x_{\emptyset}, &
  e_t &= \tfrac12 x_t - \tfrac14 x_{\emptyset}.
\end{align*}
From $x_t + x_s = 1 + x_{\emptyset} - w_0$, it follows that $e_s + e_t
= \frac12(1 - w_0)$, and hence that $ e_S = \frac12(1 +
w_0) - e_{\emptyset} = \Av(\Span{w_0}) - \Av(W)$, as required.
\end{proof}

As an immediate consequence we obtain the character of the top component of
$\kW$.

\begin{Corollary}\label{cor:PhiS}
  The $W$-module $E_{[S]}$ affords the character $\Phi_S =
  \Ind_{\Span{w_0}}^W(1) - 1_S$.
\end{Corollary}

Next we identify linear characters of centralizers of cuspidal elements.
Note that the group $W$ consists of $m$ reflections and $m$ rotations.
The centralizer of a rotation $w$ is the rotation subgroup $W^+ =
\Span{st}$ of $W$, unless $w$ is central in $W$.  The cuspidal classes
of $W$ are exactly the classes of nontrivial rotations, represented by
the set $\{(st)^j : j = 1, \dots, \Floor{\tfrac{m}2}\}$,
containing $w_0 = (st)^{m/2}$ in case $m$ is even.
The group $W^+$ is a cyclic group of order $m$ and it has $m$ linear
characters $\chi_j$, $j = 0,\dots,m-1$, defined by
\begin{align*}
  \chi_j(st) = \zeta_m^j
\end{align*}
for a primitive $m$th root of unity $\zeta_m$.  In the following
arguments, we make frequent use of the fact that
the sum of all the  nontrivial characters $\chi_j$ of $W^{+}$
equals the difference of its regular and its trivial character,
\begin{align*}
  \sum_{j=1}^{m-1} \chi_j = \Ind_{\{1\}}^{W^{+}}(1) - 1_{W^{+}},
\end{align*}
which obviously follows from $\sum_{j=0}^{m-1} \chi_j = \Ind_{\{1\}}^{W^{+}}(1)$
and $\chi_0 = 1_{W^{+}}$.

We distinguish two cases, depending on the parity of $m$.

\begin{Proposition}\label{pro:i2cus0}
  Suppose that $m = 2k$ with $k > 0$.  Let
  \begin{align*}
    \varphi_{(st)^j} = 
    \begin{cases}
      \chi_{2j}, &  0 < j < k, \\
      \epsilon_S, &  j = k.
    \end{cases}
  \end{align*}
  Then $\varphi_{(st)^j}$ is a linear character of $C_W((st)^j)$, for $j =
  1, \dots, k$, and
  \begin{align*}
    \sum_{j=1}^k \Ind_{C_W((st)^j)}^W(\varphi_{(st)^j}) = \epsilon_S +
    \sum_{j=1}^{k-1} \Ind_{W^+}^W(\chi_{2j}) = \Phi_S.
  \end{align*}
\end{Proposition}

\begin{proof} Note that $C_W((st)^j) = W^+$ and $w_0$ lies in the kernel of the
  characters $\varphi_{(st)^j} = \chi_{2j}$, for all $j = 1,
  \dots,k-1$.  Hence the $\chi_{2j}$ can be regarded as
a full set of nontrivial irreducible characters of
  the quotient group $W^+/\Span{w_0}$, whence 
  their sum $\sum_{j=1}^{k-1}
  \chi_{2j}$  equals the difference of its regular and its
  trivial characters.  Thus, as a character of $W^+$, we
  have
\begin{align*}
  \sum_{j=1}^{k-1} \chi_{2j} = \Ind_{\Span{w_0}}^{W^+}(1) - 1_{W^+}.
\end{align*}
Thus
\begin{align*}
\epsilon_S +   \Ind_{W^+}^W \Bigl( \sum_{j=1}^{k-1} \chi_{2j}\Bigr) 
= \epsilon_S + \Ind_{\Span{w_0}}^W(1) - \Ind_{W^+}^W(1)
= \Ind_{\Span{w_0}}^W(1) - 1_S
= \Phi_S,
\end{align*}
where the penultimate equality
follows from 
$\Ind_{W^+}^W(1) = 1_S + \epsilon_S$.
\end{proof}

\begin{Proposition}\label{pro:i2cus1}
Suppose that $m = 2k+1$ for some $k > 0$.
  For $j = 1, \dots, k$, let
  \begin{align*}
    \varphi_{(st)^j} = \chi_j.
  \end{align*}
Then $\varphi_{(st)^j}$ is a linear character of $C_W((st)^j)$,
for $j = 1, \dots, k$, and
  \begin{align*}
\sum_{j=1}^k \Ind_{C_W((st)^j)}^W(\varphi_{(st)^j})
=
\sum_{j=1}^{k} \Ind_{W^+}^W(\chi_j) =
    \Phi_S.
  \end{align*}
\end{Proposition}

\begin{proof}
  We have $C_W((st)^j) = W^+$
and $\Res^W_{W^+}(\Ind_{W^+}^W(\chi_j)) = \chi_j + \chi_{m-j}$
 for all $j = 1, \dots,k$.  Hence
\begin{align*}
\Res^W_{W^+}\Bigl(\sum_{j=1}^k \Ind_{W^+}^W(\chi_j)\Bigr)
&= 
\sum_{j=1}^{m-1} \chi_j
= 
\Ind_{\{1\}}^{W^+}(1) - 1_{W^+} 
\\ &
= 
\Res^W_{W^+} (\Ind_{\Span{w_0}}^W(1) - 1_S)
= 
  \Res^W_{W^+} (\Phi_S).
\end{align*}
It follows that
\begin{align*}
   \Phi_S = \sum_{j=1}^k \Ind_{W^+}^W(\chi_j),
\end{align*}
 since the restrictions of
both characters to the subgroup $\Span{w_0}$ of $W$
also coincide.
\end{proof}

\begin{Proposition}\label{pro:PsiS}
  Let $\pi_{\AA}$ be the character of the permutation action of $W$ on
  the  hyperplane arrangement $\AA$.  Then  $W$  acts on  the degree  $1$
  component of $A(W)$ with character $\pi_{\AA}$, and
  $W$ acts on the component $A_{[S]}$ of $A(W)$ with character
  \begin{align*}
    \Psi_S =  \pi_{\AA} - 1_S.
  \end{align*}
  Consequently, $W$ acts on $A(W)$ with character $2 \pi_{\AA}$.
\end{Proposition}

\begin{proof}
The degree $1$ component of $A(W)$ has basis $\{a_t : t \in T\}$
and $W$ acts on it by permuting the basis vectors.
In order to analyze the top component of $A(W)$, we make this permutation
action explicit as follows.
  
Label the hyperplanes $H_0, \dots, H_{m-1}$, so that the hyperplane $H_j$
is spanned by $\zeta_{2m}^j$, where $\zeta_{2m} = e^{2\pi i/2m}$ is
a primitive $2m$th root of unity,
as shown in Figures~\ref{fig:i2odd} and~\ref{fig:i2even}.

Let $s$ be the reflection about $H_0$ (the $x$-axis) and $ts =
(st)^{-1}$ the (anti-clockwise) rotation about the angle $2 \pi /m$.
Then $t$ is the reflection about $H_{m-1}$.

The reflection $s$ then permutes the hyperplanes according to the rule
\begin{align*}
  H_j.s = H_{m-j},
\end{align*}
for $j = 0, \dots,m-1$, fixing $H_0$.
The rotation $ts$ acts as 
\begin{align*}
  H_j.ts = H_{j+2},
\end{align*}
for $j = 0, \dots,m-1$, where the indices are reduced mod $m$ if necessary.

The top  component $A_{[S]}$ has  a basis $\{a_0  a_j : j =  1, \dots,
m-1\}$, where $W$  acts on the indices as  indicated above, subject to
the relation $a_0 a_j - a_0 a_k + a_j a_k = 0$, i.e.,
\begin{align*}
a_j a_k  = a_0  a_k - a_0 a_j.
\end{align*}

The reflection $s$ fixes $H_0$ and thus maps $a_0 a_j$ to
\begin{align*}
  a_0 a_j.s = a_0 a_{m-j},
\end{align*}
for $j = 1, \dots, m-1$.
The rotation $ts$ maps $a_0 a_j$ to
\begin{align*}
  a_0 a_j.ts = a_2 a_{j+2} = 
  \begin{cases}
a_0 a_{j+2} - a_0 a_2, & j \neq m - 2, \\
-a_0 a_2, & j = m - 2.
  \end{cases}
\end{align*}
Now define vectors
\begin{align*}
  b_0 = -\frac1m\sum_{j=1}^{m-1} a_0 a_j
\end{align*}
and, for $j = 1, \dots, m-1$,
\begin{align*}
  b_j = a_0 a_j + b_0.
\end{align*}
Then $b_0.s =  b_0$ and $b_j.s = b_{m-j}$ for $j  = 1,\dots, m-1$.
Moreover,
$b_j.ts = b_{j+2}$ for $j =
0,\dots, m-1$, with indices reduced mod $m$ if necessary.  Hence the map $a_j \mapsto b_j$ is a $W$-equivariant bijection
from the basis $\{a_j : j = 0,\dots, m-1\}$ of the degree $1$ component
to a generating set $\{b_j : j = 0,\dots, m-1\}$ of $A_{[S]}$,
and since $\sum_{j=0}^{m-1} b_j = 0$ in $A_{[S]}$,  the character
of $W$ on $A_{[S]}$ is $\pi_{\AA} - 1_S$.
\end{proof}

\begin{Lemma} \label{la:aL-spans}
  The element $a_0 a_{m-1}$ generates the top component $A_{[S]}$ as
  $\C W$-module.
\end{Lemma}

\begin{proof}
  Let $M = a_0 a_{m-1}.\C W$.  Then $M$ contains the elements
  \[
  a_0a_1 = a_0 a_{m-1}.s, \quad a_1 a_2 = -a_0a_{m-1}.ts,
  \quad\text{and}\quad a_0a_2 = a_0a_1 + a_1a_2,
  \]
  and, by induction, the elements
  \[
  a_{j-1} a_j = a_{j-3} a_{j-2}.ts, \quad\text{and}\quad a_0 a_j = a_0
  a_{j-1} + a_{j-1} a_j,
  \]
  for $j > 2$. Consequently, $M$ contains the basis $\{a_0 a_j : j = 1,
  \dots, m-1\}$ of $A_{[S]}$, whence $M = A_{[S]}$.
\end{proof}

\begin{Proposition}\label{pro:psi=phi.e}
$\Psi_S = \Phi_S \epsilon_S$.  
\end{Proposition}

\begin{proof}
  We  distinguish two cases.

  If $m$  is odd, then $\pi_{\AA} = \Ind_{\Span{s}}^W(1)$, since
$C_W(s) = \Span{s}$ and all reflections are conjugates of~$s$.
Hence
\begin{align*}
  \Psi_S 
=  \Ind_{\Span{s}}^W(1) - 1_S
=  \Ind_{\Span{w_0}}^W(1) - 1_S
= \Phi_S
\end{align*}
and $\Phi_S = \Phi_S \epsilon_S$, since
$\Phi_S(w) = 0$ for all $w \in W$ with $\epsilon_S(w) = -1$.

If $m$ is even, then $\Ind_{\Span{w_0}}^W(1) \epsilon_S = \Ind_{\Span{w_0}}^W(1)$
and 
\begin{align*}
\Phi_S \epsilon_S = (\Ind_{\Span{w_0}}^W(1) - 1_S) \epsilon_S
= \Ind_{\Span{w_0}}^W(1) - \epsilon_S
= \pi_{\AA} - 1_S
= \Psi_S,
\end{align*}
since $\pi_{\AA} - \Ind_{\Span{w_0}}^W(1) = 1_S - \epsilon_S$, as can be 
easily verified.
\end{proof}

We can now conclude that Conjecture~\ref{conj:b} holds for $W$ of rank $2$.
 
\begin{Theorem}\label{thm:BforDihedral}
  Let $W$ be a Coxeter group of rank $2$, generated by $S = \{s, t\}$.
  Then, with notation as above, the top component characters of $W$
  are $\Phi_S = \Ind_{\Span{w_0}}^W(1) - 1_S$ and $\Psi_S = \pi_{\AA}
  - 1_S = \Phi_S \epsilon_S$.  Moreover, $W$ satisfies
  Conjecture~\ref{conj:b} with $\varphi_{(st)^j} = \chi_j$ in case $m$
  odd, while $\varphi_{w_0} = \epsilon_S$ and $\varphi_{(st)^j} =
  \chi_{2j}$ in case $m$ even.
\end{Theorem}

\begin{proof}
  Apply Propositions \ref{pro:i2cus0}, \ref{pro:i2cus1}, and
  \ref{pro:psi=phi.e}, and Remark \ref{rem:comc}.
\end{proof}

\begin{Corollary}\label{cor:rank2}
  Suppose that $W$ is a Coxeter group with rank at most $2$. Then Conjecture
  \ref{conj:a} holds for $W$.
\end{Corollary}

\begin{proof}
  By Lemmas \ref{la:a0} and \ref{la:a1}, and Theorem
  \ref{thm:BforDihedral}, Conjecture \ref{conj:b} holds for all parabolic
  subgroups of $W$. By Theorem \ref{thm:c=>a} it suffices to show that
  Conjecture \ref{conj:c} holds for all subsets $L\subseteq S$. If
  $|L|=0,1$, this follows from Corollary \ref{cor:0,1}. It follows from
  Theorem \ref{thm:BforDihedral} that Conjecture \ref{conj:c} holds when the
  rank of $W$ and $|L|$ are both equal $2$.
\end{proof}

It follows in particular from Corollary \ref{cor:rank2} that every Coxeter
group of type $I_2(m)$ satisfies Conjecture~\ref{conj:a}.  We list the
corresponding decomposition of the regular character $\rho_W$ into
characters $\Phi_{[L]} = \Ind_{N_W(W_L)}^W \widetilde{\Phi}_L$ and the
decomposition of the Orlik-Solomon character $\omega_W$ into characters
$\Psi_{[L]} = \Ind_{N_W(W_L)}^W \widetilde{\Psi}_L$ in
Table~\ref{tab:phi} below. 
\begin{table}[htbp]
  \centering
\begin{align*}
  \begin{array}{c|ccccc}
    \hline
 & 1 & s & t & w_0 &(st)^i \\ \hline\hline
\Phi_{[\emptyset]} & 1 & 1 & 1 & 1 & 1 \\
\Phi_{[\{s\}]} & k & .\mid 1 & .\mid {-1} & -k & .  \\
\Phi_{[\{t\}]} & k & .\mid {-1} & .\mid 1 & -k & . \\
\Phi_{[S]} & m-1 & -1 & -1 & m-1 & -1 \\ \hline
\rho_W & 2m & . & . & . & . \\ \hline
\Psi_{[\emptyset]} & 1 & 1 & 1 & 1 & 1 \\
\Psi_{[\{s\}]} & k & 2 \mid 1 & . \mid 1 & k & .  \\
\Psi_{[\{t\}]} & k & . \mid 1 & 2 \mid 1 & k & . \\
\Psi_{[S]} & m-1 & 1 & 1 & m-1 & -1 \\ \hline
\omega_W & 2m & 4 & 4 & 2m & . \\ \hline
  \end{array}
\qquad
  \begin{array}{c|ccc}
    \hline
 & 1 & s & (st)^i \\ \hline\hline
\Phi_{[\emptyset]} & 1 & 1 & 1 \\
\Phi_{[\{s\}]} & m & -1 & . \\
\Phi_{[S]} & m-1 & . & -1 \\ \hline
\rho_W & 2m & . & . \\ \hline
\Psi_{[\emptyset]} & 1 & 1 & 1 \\
\Psi_{[\{s\}]} & m & 1 & . \\
\Psi_{[S]} & m-1 & . & -1 \\ \hline
\omega_W & 2m & 2 & .\\ \hline
  \end{array}
\end{align*}
\caption{The characters $\Phi_{\lambda}$ and $\Psi_{\lambda}$ for
    $I_2(m)$; $m=2k$, $m=2k+1$.}
  \label{tab:phi}
\end{table}
In Table~\ref{tab:phi}, the left character table covers the case $m = 2k$
and the right character table covers the case $m = 2k+1$.  The columns of
the character tables are labelled by representatives of the conjugacy
classes of $W$, where the parameter in $(st)^i$ is $i = 1, \dots, k-1$ for
$m = 2k$, and $i = 1, \dots, k$ for $m = 2k+1$.  An entry `$.$' in the table
stands for the value $0$.  As observed in Proposition~\ref{pro:PsiS}, the
rank $1$ component of $\omega_W$ is the permutation character of the action
of $W$ on the set $\AA$ of hyperplanes.  In case $m = 2k$, the constituent
$\Psi_{[\{s\}]}$ corresponds to the action on the $W$-orbit of the
hyperplane $H_s$, and whether the element $s$ has $2$ or $1$ fixed points in
this action depends on whether $k$ is even or odd. In such a situation, an
entry of the form `$x \mid y$' in the table stands for `$x$ if $k$ is even
and $y$ if $k$ is odd'.

We saw in Theorem \ref{thm:BforDihedral} that Conjecture \ref{conj:b} holds
when $W$ has rank $2$ and we saw in Corollary \ref{cor:0,1} that Conjecture
\ref{conj:c} holds when the subset $L\subseteq S$ has size $\Size{L} \leq
1$. In the rest of this section, we prove that if the
parabolic subgroup $W_L$ has rank two, then Conjecture~\ref{conj:c} holds
for any ambient group $W$. A similar result when $W_L$ is a product of
symmetric groups would reduce the proof of Conjecture \ref{conj:a} to
considering only a small number of cases.

From now on, $W$ is a finite Coxeter group, generated by $S$ with
$|S|\geq 3$ and $W_L$ is a rank $2$ parabolic subgroup of $W$ with $L
= \{s, t\} \subseteq S$.  The elements $x_K$ and $e_K$ are defined
relative to the ambient set $S$.  We use a superscript to indicate
this ambient set when it is not equal to $S$.
Thus, for $K \subseteq L$, $x_K^L$  denotes a basis element
of the descent algebra of $W_L$.

If $W_L$ is bulky, then $W_L$ satisfies Conjecture~\ref{conj:c}, by
Theorem~\ref{thm:bulky}.

If $W_L$ is not bulky, then $N_L$ does not centralize $W_L$ and so $N_L$
contains an element inducing the nontrivial graph automorphism $\gamma$ on
$W_L$, interchanging $s$ and $t$. In this case $s$ and $t$ are conjugate in
$W$ and so $W_L$ is either of type $A_1 \times A_1$ or of type $I_2(m)$ for
odd $m > 2$. We distinguish two cases accordingly.

First, suppose that $W_L$ is of type $A_1 \times A_1$.  Then $W_L$ has
exactly one cuspidal element $w = st = ts$, which is central in $W_L$ and
invariant under $N_L$, hence central in $N_W(W_L)$.  We have
\[
\varphi_w = \Phi_L = \epsilon_L,\quad\text{and} \quad \psi_w = \Psi_L = 1_L,
\]
by Corollary~\ref{cor:PhiS} and Proposition~\ref{pro:PsiS}.  Parts (i) and
(ii) of Conjecture~\ref{conj:c} are therefore trivially satisfied, with
\[
\widetilde{\varphi}_w = \widetilde{\Phi}_L, \quad\text{and}\quad
\widetilde{\psi}_w = \widetilde{\Psi}_L,
\]
which exist by Propositions~\ref{pro:Phi} and~\ref{pro:Psi}.

For part (iii) of Conjecture~\ref{conj:c}, 
note that the idempotent
\begin{align*}
  f = \tfrac14 (1 - s - t + st)
\end{align*}
spans a subspace of $\C W_L$ affording the character $\Phi_L$. As in the proof of Lemma~\ref{la:eS}, 
\begin{align*}
  e_L^{L} = 1 - \tfrac12 x_s^L - \tfrac12 x_t^L + \tfrac14 x_{\emptyset}^L
= \tfrac14 (1 + st) - \tfrac14 (s + t) = f,
\end{align*}
and thus $e_L^L f = e_L^L$ is a basis of the top component of $W_L$ which is
centralized by $N_L$.  Hence $\widetilde{\varphi}_w(un) = \varphi_w(u)$, for
$u \in W_L$ and $n \in N_L$.  Moreover, note that $a_L = a_s a_t$ spans the
top component of $A(W_L)$, and that $e_L n = e_L$, whereas $a_L .n =
\sigma_L(n) a_L$ for $n \in N_L$.  It follows that $\widetilde{\psi}_L(un) =
\psi_L(u) \sigma_L(n) = \varphi_L(u) \epsilon(u) \epsilon(n) \alpha_L(n) =
\widetilde{\varphi}_L(un) \epsilon(un) \alpha_L(un)$, for $u \in W_L$ and $n
\in N_L$, as desired.  This proves the
following proposition.

\begin{Proposition}\label{pro:CforA1A1}
  Suppose $L = \{s,t\} \subseteq S$ is such that $W_L$ is of type $A_1\times
  A_1$. Then Conjecture~\ref{conj:c} holds for $L\subseteq S$.
\end{Proposition}

Second, suppose that $W_L$ is of type $I_2(m)$ for $m$ odd. 
Recall that the character $\chi_j \colon st \mapsto \zeta_m^j$ is afforded by the subspace of $\C W^+$ spanned by the idempotent
\begin{align}\label{eq:fj}
  f_j &= \frac1m\sum_{k=0}^{m-1} \zeta_m^{jk} (st)^{-k},
\end{align}
for $j = 1, \dots, m-1$.
As usual, denote by $w_L$ the longest element of $W_L$.
Note that $f_j^{w_L} = f_{m-j}$, for $j = 1, \dots, m-1$, since $(st)^{w_L} = (st)^{-1}$, and that
\begin{align*}
  e_L^L f_j = \Av(\Span{w_L}) f_j,
\end{align*}
by Lemma~\ref{la:eS}, since $\Av(W_L) f_j =
\sum_{k=0}^{m-1} \zeta_m^{jk} \Av(W_L) = 0$,
for $j = 1, \dots, m-1$. 

Obviously, the graph automorphism $\gamma$ swaps $e_L^L f_j$ and $e_L^L f_{m-j}$,
and so does right multiplication by $w_L$:
\begin{align*}
    e_L^L f_j w_L &
=  \Av(\Span{w_L}) f_j w_L 
=  \Av(\Span{w_L}) w_L\, f_j^{w_L}  \\&
=  \Av(\Span{w_L}) f_j^{w_L} 
=  \Av(\Span{w_L}) f_{m-j} 
=    e_L^L f_{m-j}.
\end{align*}
Moreover, if $n \in N_L$ induces the automorphism $\gamma$ on $W_L$,
then $w_L n \in C_W(st)$.  Therefore, if we write $N_L = N_L^+ \cup
N_L^-$, where $N_L^+ = N_L \cap C_W(st)$ and $N_L^- = N_L \setminus
N_L^+$, then we have
\begin{align*}
 C_W(st) = C_{W_L}(st) N_L^+ \cup C_{W_L}(st) w_L N_L^-.
\end{align*}
It follows that we can naturally extend the characters $\varphi_{(st)^j}$
to characters  $\widetilde{\varphi}_{(st)^j}$ of the full centralizer $C_W(st)$ 
via
\begin{align*}
  \widetilde{\varphi}_{(st)^j}(c) = \varphi_{(st)^j}(v),
\end{align*}
where either $c = vn$ for some $v \in C_{W_L}(st)$ and $n \in N_L^+$,
or $c = v w_L n$ for
some $v \in C_{W_L}(st)$ and $n \in N_L^-$.

We are now in a position to prove that Conjecture~\ref{conj:c} holds in this
case.

\begin{Proposition}\label{pro:CforI2(2k+1)}
  Suppose $L = \{s,t\} \subseteq S$ is such that the order $m$ of $st$ is
  odd.  Then Conjecture~\ref{conj:c} holds for $L\subseteq S$.
\end{Proposition}

\begin{proof}
  We have that $\Phi_L = \sum_{j=1}^k \Ind_{C_{W_L}(st)}^{W_L}
  \varphi_{(st)^j}$, by Theorem~\ref{thm:BforDihedral}.

Recall from \cite[Sec.~3]{DouglassPfeiffer2011}  that  
left multiplication by $x_L$ defines an isomorphism of the right
$W_L$-modules $e_L^L \C W_L$ and $e_L \C W_L$.
Therefore, the elements $e_L f_j = x_L e_L^L f_j$, for $j = 1, \dots, m{-}1$, form a
$\C$-basis of 
$e_L \C W_L$, which as  $N_W(W_L)$-module affords the character
$\widetilde{\Phi}_L$, and as $W_L$-module is isomorphic to the top
component $E_L$ with character $\Phi_L$.

Moreover, if we denote $M_j = e_L f_j \C N_W(W_L)$, then the
$N_W(W_L)$-module $M_j$ has $\C$-basis $\{e_L f_j, e_L f_{m-j}\}$, due
to the nontrivial action of $\gamma$ and $w_L$, and the direct sum
$\bigoplus_{j=1}^k M_j$ is a decomposition of $e_L \C W_L$ as
$N_W(W_L)$-module.
Consequently,
part (i) of Conjecture~\ref{conj:c}  follows from the observation
that
as $W_L$-module  $M_j$  affords the character $\Ind_{C_{W_L}(st)}^{W_L} \varphi_{(st)^j}$ and
as $N_W(W_L)$-module it affords the character $\Ind_{C_{W}(st)}^{N_W(W_L)} \widetilde{\varphi}_{(st)^j}$, i.e.,
\begin{align*}
  \widetilde{\Phi}_L = \sum_{j=1}^k \Ind_{C_W(st)}^{N_W(W_L)}
  \widetilde{\varphi}_{(st)^j}.
\end{align*}

  By Remark \ref{rem:comc}, it now suffices to show that
  $\widetilde{\Psi}_L= \widetilde{\Phi}_L \epsilon_S \alpha_L$. For this,
  denote $a_L = a_s a_t$, and recall from Lemma~\ref{la:aL-spans} that $a_L
  \C W_L$ is isomorphic to the top component of $A(W_L)$. Since $m$ is odd,
  we have $W_L = \Span{st} \cup w_L \Span{st}$ and thus
  \begin{align*}
    a_L \C \Span{st} = a_L \C W_L,
  \end{align*}
  since $a_L w_L = a_s a_t.w_L = a_t a_s = - a_s a_t = -a_L$.

Since the idempotents $f_j$ from equation (\ref{eq:fj})
form a Wedderburn basis of the group algebra $\C \Span{st}$, 
the module $a_L \C \Span{st}$ is spanned by the elements
$\{a_L f_j : j = 0, \dots, m-1\}$, and
since
\begin{align*}
   a_L f_0 
&= \sum_{k=0}^{m-1} a_0 a_{m-1} .(ts)^k
= \sum_{k=0}^{m-1} a_{k+1} a_k
\\&= a_0 a_{m-1} - a_0 a_1 + \sum_{k=1}^{m-2} a_0  a_k - a_0 a_{k+1}
= 0,
\end{align*}
 we also have that the set $\{a_L f_j : j = 1, \dots, m-1\}$
is a $\C$-basis of $a_L \C W_L$. 
Conjecture~\ref{conj:c} (iii) now follows if we can show that
\begin{align}\label{eq:final}
  a_L f_j.w = \epsilon(w) \alpha_L(w) e_L f_j w,
\end{align}
for all $w \in N_W(W_L)$. It suffices  to show this for $w = st$, $w =
w_L$, and for $w = n \in N_L$.

For $w = st$, (\ref{eq:final}) follows, since $f_j st = \zeta_m^j f_j$ and
$\epsilon(st) = \alpha_L(st) = 1$.
For $w = w_L$, (\ref{eq:final}) follows, since $f_j w_L = w_L f_{m-j}$ and $e_L w_L
= e_L$, $a_L w_L =  -a_L$, and $\epsilon(w_L) = -1$ and $\alpha_L(w_L)
= 1$.
Finally, for $w  = n \in N_L$ (\ref{eq:final})   holds, since $f_j n =  n f_j^n$ and
$e_L  n =  e_L$, $a_L  n =  
\sigma_L(n) a_L$ and $\sigma_L(n) =  \epsilon(n)\alpha_L(n)$, by
Lemma~\ref{la:alpha-sigma}.
\end{proof}

We summarize Propositions~\ref{pro:CforA1A1}, \ref{pro:CforI2(2k+1)}, and
Theorem~\ref{thm:bulky} for rank $2$ parabolic subgroups as follows.

\begin{Theorem}\label{thm:CforDihedral}
  Suppose that $W_L$ is a rank $2$ parabolic subgroup of $W$.
  Then Conjecture~\ref{conj:c} holds for $W_L$.
\end{Theorem}

\bigskip 

{\bf Acknowledgments}: The authors acknowledge the financial
support of the DFG-priority programme SPP1489 ``Algorithmic and Experimental
Methods in Algebra, Geometry, and Number Theory''.  Part of the research for
this paper was carried out while the authors were staying at the
Mathematical Research Institute Oberwolfach supported by the ``Research in
Pairs'' programme in 2010.
The second author wishes to acknowledge support from Science Foundation
Ireland.

\bibliographystyle{plain}
\bibliography{dihedral}

\def\cprime{$'$}
\begin{thebibliography}{10}

\bibitem{BBGarsia1990}
F.~Bergeron, N.~Bergeron, and A.~M. Garsia.
\newblock Idempotents for the free {L}ie algebra and {$q$}-enumeration.
\newblock In {\em Invariant theory and tableaux ({M}inneapolis, {MN}, 1988)},
  volume~19 of {\em IMA Vol. Math. Appl.}, pages 166--190. Springer, New York,
  1990.

\bibitem{BergeronEtAl1992}
F.~Bergeron, N.~Bergeron, R.~B. Howlett, and D.~E. Taylor.
\newblock A decomposition of the descent algebra of a finite {C}oxeter group.
\newblock {\em J. Algebraic Combin.}, 1(1):23--44, 1992.

\bibitem{Brieskorn1973}
E.~Brieskorn.
\newblock Sur les groupes de tresses [d'apr\`es {V}. {I}. {A}rnol\cprime d].
\newblock In {\em S\'eminaire {B}ourbaki, 24\`eme ann\'ee (1971/1972), {E}xp.
  {N}o. 401}, pages 21--44. Lecture Notes in Math., Vol. 317. Springer, Berlin,
  1973.

\bibitem{DouglassPfeiffer2011}
J.~M. Douglass, G.~Pfeiffer, and G.~R{\"o}hrle.
\newblock Coxeter arrangements and {S}olomon's descent algebra.
\newblock arxiv:1101:2075, 2011.

\bibitem{geckpfeiffer2000}
M.~Geck and G.~Pfeiffer.
\newblock {\em Characters of finite {C}oxeter groups and {I}wahori-{H}ecke
  algebras}, volume~21 of {\em London Mathematical Society Monographs. New
  Series}.
\newblock The Clarendon Press Oxford University Press, New York, 2000.

\bibitem{Hanlon1990}
P.~Hanlon.
\newblock The action of {$S_n$} on the components of the {H}odge decomposition
  of {H}ochschild homology.
\newblock {\em Michigan Math. J.}, 37(1):105--124, 1990.

\bibitem{Howlett1980}
R.~B. Howlett.
\newblock Normalizers of parabolic subgroups of reflection groups.
\newblock {\em J. London Math. Soc. (2)}, 21:62--80, 1980.

\bibitem{KonvalinkaPfeiffer2010}
M.~Konvalinka, G.~Pfeiffer, and C.~R{\"o}ver.
\newblock A note on element centralizers in finite {Coxeter} groups.
\newblock {\em J. Group Theory}, 2011.
\newblock doi:10.1515/JGT.2011.074, arxiv:1005:1186.

\bibitem{LehrerSolomon1986}
G.~I. Lehrer and L.~Solomon.
\newblock On the action of the symmetric group on the cohomology of the
  complement of its reflecting hyperplanes.
\newblock {\em J. Algebra}, 104(2):410--424, 1986.

\bibitem{OrlikSolomon1983}
P.~Orlik and L.~Solomon.
\newblock Coxeter arrangements.
\newblock In {\em Singularities, {P}art 2 ({A}rcata, {C}alif., 1981)},
  volume~40 of {\em Proc. Sympos. Pure Math.}, pages 269--291. Amer. Math.
  Soc., Providence, RI, 1983.

\bibitem{OrlikTerao1992}
P.~Orlik and H.~Terao.
\newblock {\em Arrangements of hyperplanes}, volume 300 of {\em Grundlehren der
  Mathematischen Wissenschaften}.
\newblock Springer-Verlag, Berlin, 1992.

\bibitem{Pfeiffer2009}
G.~Pfeiffer.
\newblock A quiver presentation for {S}olomon's descent algebra.
\newblock {\em Adv. Math.}, 220(5):1428--1465, 2009.

\bibitem{PfeifferRoehrle2005}
G.~Pfeiffer and G.~R{\"o}hrle.
\newblock Special involutions and bulky parabolic subgroups in finite {C}oxeter
  groups.
\newblock {\em J. Aust. Math. Soc.}, 79(1):141--147, 2005.

\bibitem{Schocker2001}
M.~Schocker.
\newblock {\"Uber die h\"oheren Lie-Darstellungen der symmetrischen Gruppen}.
\newblock {\em Bayreuth. Math. Schr.}, 63:103--263, 2001.

\bibitem{Solomon1976}
L.~Solomon.
\newblock A {M}ackey formula in the group ring of a {C}oxeter group.
\newblock {\em J. Algebra}, 41(2):255--264, 1976.

\end{thebibliography}
\end{document}